\documentclass[11pt]{amsart}
\usepackage[a4paper,hmargin=3cm,vmargin=3cm]{geometry}
\usepackage{amsfonts,amssymb,amscd,amstext,amsthm}
\usepackage[pdftex]{graphicx}
\usepackage[dvips]{epsfig}
\usepackage[T1]{fontenc}
\usepackage[english]{babel}
\usepackage[utf8]{inputenc}
\usepackage[bookmarksnumbered,plainpages]{hyperref}
\usepackage{indentfirst}
\usepackage{color}
\usepackage{url}
\usepackage[shortlabels]{enumitem}
\usepackage{imakeidx}
\usepackage{titlesec}
\usepackage{mathrsfs}
\usepackage{stmaryrd}
\usepackage{textfit}
\usepackage{lipsum}
\usepackage{amsmath}
\usepackage{mathtools}
\usepackage{fancyhdr}
\usepackage{times}
\def \R{\mbox{${\mathbb R}$}}
\def \S{\mbox{${\mathbb S}$}}
\def \H{\mbox{${\mathbb H}$}}

\def \Q{\mbox{${\mathbb Q}$}}

\newcommand{\rank}{\text{rank}}
\newcommand{\co}{\colon}

\newcommand{\Iso}{\mathrm{Iso}}

\titleformat{\section}%[display]
{\filcenter\bfseries\large} {\thesection{.}}{0.2cm}{}%[$\vspace*{-1.0cm}$]
\titleformat{\subsection}[runin]
{\bfseries} {\thesubsection{.}}{0.15cm}{}[.]
\titleformat{\subsubsection}[runin]
{\em}{\thesubsubsection{.}}{0.15cm}{}[.]
\usepackage[up,bf]{caption}
\newtheorem{theorem}{Theorem}[section]
\newtheorem{lemma}[theorem]{Lemma}

\newtheorem{thm}{Theorem}

\newtheorem{claim}[theorem]{Claim}

\theoremstyle{definition}
\newtheorem{definition}[theorem]{Definition}
\newtheorem{remark}[theorem]{Remark}
 
\newcommand{\PfThm}[1]{{\noindent\em Proof of Theorem #1}.}
\newcommand{\EPf}{\hfill$\square$\newline}
\newcommand{\EPfC}{\hfill$\triangle$\newline}
\numberwithin{equation}{section}
\numberwithin{figure}{section}
\begin{document}

\fancyhead[LO]{Complete cohomogeneity one hypersurfaces of $\H^{n+1}$}
\fancyhead[RE]{Felippe Guimarães, Fernando Manfio, Carlos E. Olmos}
\fancyhead[RO,LE]{\thepage}

\thispagestyle{empty}

\begin{center}
{\bf \LARGE Complete cohomogeneity one hypersurfaces of $\H^{n+1}$}
\vspace*{5mm}

%% Authors
\hspace{0.2cm} {\Large Felippe Guimarães, Fernando Manfio and Carlos E. Olmos} 
\end{center}

%\vspace*{8mm}

\begin{quote}
{\small
\noindent {\bf Abstract}\hspace*{0.1cm}
    We study isometric immersions $f\co M^n \rightarrow \mathbb{H}^{n+1}$ into hyperbolic space of dimension $n+1$ of a complete Riemannian manifold of dimension $n$ on which a compact connected group of intrinsic isometries acts with principal orbits of codimension one. We provide a characterization if either $n \geq 3$ and $M^n$ is compact, or $n \geq 5$ and the connected components of the set where the sectional curvature is constant and equal to $-1$ are bounded.
}
\\

{\small
\noindent {\textbf{Mathematics Subject Classification:}}\hspace*{0.1cm}
    53C42, 53C40, 53C30.
}
\\
{\small
\noindent {\textbf{Keywords:}}\hspace*{0.1cm}
    Cohomogeneity one manifolds, hypersurfaces, hyperbolic space.
}

\end{quote}

\section{Introduction}

The study of homogeneous \(n\)-dimensional submanifolds \(M^n\), i.e., where the isometry group \(\text{Iso}(M^n)\) acts transitively on \(M^n\), was initiated by Kobayashi in \cite{KobayashiCompactHomRH}, where the author proved that every compact homogeneous Euclidean hypersurface must be a round sphere. The complete classification of homogeneous hypersurfaces in simply connected Riemannian manifolds with constant sectional curvature \(c \in \mathbb{R}\), denoted by \(\mathbb{Q}^{n+1}_c\), without the compactness hypothesis, was later achieved in the works \cite{NaganoTakahashiHomEuclidean,TakahashiCod1Hom,TakahashiCod1Hom3,TakagiTakahashiHomSphere,HsiangLawsonMinimalHomogeneous}. Subsequent studies relaxed the homogeneity assumption, focusing on isometric immersions \(f : M^n \to \mathbb{Q}^{n+1}_c\) where a \emph{compact}, \emph{connected} subgroup \(G\) of \(\text{Iso}(M^n)\) acts with orbits of codimension one. Such \(f\) is referred to as a \emph{hypersurface of \(G\)-cohomogeneity one}. In this work, we will always assume that \(G\) is compact unless stated otherwise. In \cite{PodestaSpiroCohomEuc}, it was shown that compact Euclidean hypersurfaces of \(G\)-cohomogeneity one with umbilical principal orbits must be rotational (cf. definition \ref{def:rotational}). Later articles, such as \cite{AspertiMercuriNoronhaRot,MercuriSeixasCohoRevol}, have introduced other conditions on the orbits, also leading to the hypersurfaces being rotational.

A more complete understanding of Euclidean hypersurfaces of \(G\)-cohomogeneity one was obtained in \cite{MercuriPodestaSeixasTojeiroCohomogeneityEuclidean}, where they dropped the restrictive assumption that principal orbits be umbilical and generalized the previous results on the topic (cf. \cite{TojeiroSurveyGManifolds} and reference therein for a discussion), and it can be stated as follows.

\begin{theorem}[Theorem 1.1 in \cite{MercuriPodestaSeixasTojeiroCohomogeneityEuclidean}]
    Let \(f\co M^n \to \mathbb{R}^{n+1}\) be a complete hypersurface of \(G\)-cohomogeneity one. Assume either that \(n\geq 3\) and \(M^n\) is compact or that \(n\geq 5\) and the connected components of the flat part of \(M^n\) are bounded. Then \(f\) is either rigid or a hypersurface of revolution.
\end{theorem}

Note that the assumption concerning the boundedness of the connected components of the flat part is necessary. Without this assumption, it is possible to construct flat cylinders over planar curves that do not meet the conclusions of the theorem. In their result, they made use of an unpublished version of the main result in \cite{SackstederHyp} due to Ferus (cf. Theorem 13.2 in \cite{base}). To be able to use such a result, they needed to understand the hypersurface of \(G\)-cohomogeneity one that has at least one complete \((n-2)\)-dimensional leaf of the relative nullity distribution, i.e., the kernel of the second fundamental form. They showed that, as in the case of homogeneous hypersurfaces, they must be isometric to \(M^n = \S^2 \times \mathbb{R}^{n-2}\) and the isometric immersion is canonical in \(\mathbb{R}^{n+1}\), i.e., as a product of a hypersurface and the Euclidean factor.

Hypersurfaces of \(G\)-cohomogeneity one in the hyperbolic space with umbilical principal orbits were studied in \cite{AspertiCaputiCohomogeneityHyperbolic}. They showed that, when \(G\) is compact, the hypersurface must be rotational. In this work, we address complete hypersurfaces of \(G\)-cohomogeneity one in hyperbolic space without assuming the umbilicity in the orbits. As these orbits will be homogeneous submanifolds with codimension two in the hyperbolic space, we employ the tools developed in \cite{spaceFormsNoronhadeCastro} to show that the main result in \cite{MercuriPodestaSeixasTojeiroCohomogeneityEuclidean}, as stated below, holds in the hyperbolic setting.

\begin{thm}\label{thm:Main}
    Let \(f\co M^n \to \mathbb{H}^{n+1}\) be a complete hypersurface of \(G\)-cohomogeneity one. Assume either that \(n\geq 3\) and \(M^n\) is compact or that \(n\geq 5\) and the connected components of the set where the sectional curvature is constant and equal to $-1$ are bounded. Then \(f\) is either rigid or a rotational submanifold.
\end{thm}

\begin{remark}
    In the case where the hypersurface is rigid in Theorem \ref{thm:Main}, it will be an extrinsically cohomogeneity one hypersurface, and such hypersurfaces are well-known in \(\mathbb{H}^{n+1}\) (cf. \cite{BerndtMartinaCohomogeneity}). When \(f\) is a rotational submanifold, it will be the orbit of a polar action whose section contains the profile submanifold (cf. \cite{BerndtConsoleOlmosBook,gorodski2022topics}).
\end{remark}

Similar to the work that inspired this one, it is necessary to understand the hypersurfaces that admit a complete leaf of dimension \(n-2\) from the distribution of relative nullity. Unlike the homogeneous case, where such hypersurfaces do not exist as shown in \cite{TakahashiCod1Hom}, examples appear in the scenario of hypersurfaces with cohomogeneity one. These are necessarily rotational submanifolds, not over a profile curve, but over a 3-dimensional Riemannian manifold as the profile.

\begin{thm} \label{thm:Rothyper}
    Let \(f\co M^n \rightarrow \mathbb{H}^{n+1}\), with \(n \geq 5\), be a complete hypersurface of \(G\)-cohomogeneity one. If there exists a complete leaf of relative nullity with dimension \(n-2\), then \(M^n\) is isometric to \(\R \times_{\rho_0} \S^2_{c_1} \times_{\rho_1} \S^{n-3}_{c_2}\), for some constants \(c_1 > 0\) and \(c_2 > 0\) and warped functions \(\rho_0\) and \(\rho_1\) defined on \(\R\). Additionally, there exists a warped product representation \(\Psi\co V^{1+i_0} \times_{\sigma_0} \mathbb{Q}^{2+i_1}_{\tilde{c_1}} \times_{\sigma_1} \S^{n-3}_{c_2} \rightarrow \mathbb{H}^{n+1}\), where \(i_0 + i_1 = 1\) and \(V^{1+i_0}\) is an open subset of a totally geodesic submanifold, with warping functions \(\sigma_0\) and \(\sigma_1\) defined on \(V^{1+i_0}\), in which \(f\) admits an extrinsic warped product structure. More precisely, there exist isometric immersions \(f_0\co \mathbb{R} \rightarrow V^{1+i_0}\) and \(f_1\co \S^2_{c_1} \rightarrow \mathbb{Q}^{2+i_1}_{\tilde{c_1}}\) such that \(\rho_0 = \sigma_0 \circ f_0\), \(\rho_1 = \sigma_1 \circ f_1\), and \(f = \Psi \circ (f_0 \times f_1 \times Id)\), where \(Id\) is the identity map defined on \(\S^{n-3}_{c_1}\). Moreover, \(\nu = n - 2\) throughout the submanifold and the leaves of relative nullity are given by \(\mathbb{R} \times_{\rho_0} \{p\} \times_{\rho_1} \mathbb{S}^{n-3}_{c_2}\) for \( p \in \mathbb{S}^2_{c_1}\). In particular, \(f\) is a rotational submanifold.
\end{thm}

The manuscript is organized as follows: in section \ref{sec:Pre} we introduce the notation to be used, as well as results from the literature that play a key role in this work. In section \ref{sec:Proofs}, we present the proofs of the main theorems.
\section{Preliminaries}\label{sec:Pre}

Let $f: M^n \to \Q^{n+p}_c$, where $c \in \R$ and $n \geq 2$, be an isometric immersion of a Riemannian manifold. For a point $x \in M^n$ and a normal vector $\xi \in T_xM^\perp$, we denote by $\alpha_f(x)$ the second fundamental form of $f$, and by $A^f_\xi(x)$ the shape operator of $f$ with respect to $\xi$, respectively. These are related by the equation
\[
\langle A^f_\xi X,Y\rangle = \langle\alpha_f(X,Y),\xi\rangle, \ \text{for all} \ X,Y \in T_xM.
\]
\noindent The \emph{relative nullity subspace} $\Delta\left(x\right) \subset T_xM$ of $f$ at $x \in M^n$ is the kernel of its second fundamental form at $x \in M^n$, specifically,
\begin{equation*}
    \Delta \left(x\right)=\left\{X\in T_xM:\alpha_f\left(X,Y\right)=0, \ \text{ for all }Y
    \in T_xM\right\}.
\end{equation*}
The dimension $\nu(x)$ of the subspace $\Delta(x)$ is called the \emph{index of relative nullity} at $x \in M^n$. 
%Define $\nu_0 := \min_{x \in M} \nu(x)$ as the minimum value of the relative nullity. 

We say that $x\in M^n$ is a {\em totally geodesic point} of $M^n$ if $\nu(x)=n$. If $f$ has constant index of relative nullity on an open subset $U\subset M^n$, then $\Delta$ is an autoparallel distribution, that means
\[
\nabla_T S \in \Gamma(\Delta), \ \text{for all } \ 
T,S\in \Gamma(\Delta).
\]
Moreover, its leaves are totally geodesic submanifolds of $M^n$ and their images under $f$ are totally geodesic submanifolds of $\Q^{n+1}_c$. The index of relative nullity $\nu$ is upper semicontinuous and, in particular, the subset $M_0 =\{x \in M^n: \nu(x)=\nu_0\}$ where $\nu$ attains its minimum value $\nu_0$ is open. Moreover, if \( M^n \) is complete, then the leaves of relative nullity are also complete (cf. \cite{FerusNullityComplete}).
%(see \cite{base}).

We are now in a position to state the version of {\em Sacksteder's theorem} that will be used in this work. Recall that $f$ is said to be
\emph{rigid} if any other isometric immersion 
$g\co M^n\to\Q^{n+p}_c$ is congruent to $f$ by an isometry of the ambient space $\mathbb{Q}^{n+p}_c$. That is, there exists an isometry $\mathcal{I} \in \text{Iso}( \mathbb{Q}^{n+p}_c)$ such that $g = \mathcal{I} \circ f$.

\begin{theorem}[Theorem 13.2 in \cite{base}] \label{thm:Sacksteder}
Let \( f\co M^n \to \Q^{n+1}_c \) be an isometric immersion of a complete Riemannian manifold with \( n \geq 3 \) and \( c \leq 0 \). If there exists no complete leaf of relative nullity of dimension \( n-1 \) or \( n-2 \), then for any other isometric immersion \( \tilde{f} \co M^n \to \Q^{n+1}_c \), the following holds:
\begin{enumerate}[(i)]
    \item The set \( B^f \) of totally geodesic points of \( f \) coincides with that of \( \tilde{f} \).
    \item On each connected component of \( M \setminus B^f \), the shape operators of \( f \) and \( \tilde{f} \) satisfy \( A^f = \pm A^{\tilde{f}} \).
\end{enumerate}
In particular, if \( B^f \) does not disconnect \( M^n \), then \( f \) is rigid.
\end{theorem}

\begin{remark}
    Theorem \ref{thm:Sacksteder} also holds when the ambient space is the sphere \(\S^{n+1}\) without any restriction on the leaves of relative nullity, but it requires \(n\geq 4\) (cf. \cite[Lemma 32]{FloritGuimaraesCMH}).
\end{remark}

We briefly recall the notion of a warped product of isometric immersions into a space form $\Q_c^{l}$. This relies on the warped product representations of $\Q_{c}^{l}$, that is, isometries of warped products onto open subsets of $\Q_{c}^{l}$. All such isometries were described by N{\"{o}}lker \cite{NolkerWarped} for warped products with arbitrarily many factors. In particular, any isometry of a warped product with two factors onto an open subset of $\mathbb{Q}_{c}^{l}$ arises as a restriction of an explicitly constructible isometry
\[
\Psi:V^{l-m}\times_{\sigma} \mathbb{Q}_{\tilde{c}}^{m}\to \mathbb{Q}_{c}^{l}
\]
onto an open dense subset of $\mathbb{Q}_{c}^{l}$, where $\mathbb{Q}_{\tilde{c}}^{m}$ is a complete spherical submanifold of $\mathbb{Q}_{c}^{l}$ and $V^{l-m}$ is an open subset of the unique totally geodesic submanifolds $\mathbb{Q}_{c}^{l-m}$ of $\mathbb{Q}_{c}^{l}$, whose tangent space at some point $\overline{z} \in\Q_{\tilde{c}}^{m}$ is the orthogonal complement of the tangent space of $\mathbb{Q}_{\tilde{c}}^{m}$ at $\overline{z}$. The isometry $\Psi$ is completely determined by the choice of $\mathbb{Q}_{\tilde{c}}^{m}$ and of a point $\overline{z} \in\Q_{\tilde{c}}^{m}$, and it is called the {\em warped
product representation} of $\Q^l_c$ determined by these data.

\begin{definition} \label{def:rotational}
An isometric immersion 
$f\co N_0^{n_0}\times_\rho N_1^{n_1}\to\Q^{n+p}_c$ is said to be an {\em extrinsic warped product} if there exist isometric immersions $f_0\co N_0^{n_0}\to V^{n_0+l_0}$ and $f_1\co N_1^{n_1} \to \mathbb{Q}^{n_1+l_1}_{\tilde{c}}$, with $n_0+n_1=n$, and a warped product representation $\Psi\co  V^{n_0+l_0}\times_\sigma\Q^{n_1+l_1}_{\tilde{c}}\to\Q^{n+p}_c$, with $l_0+l_1=p$ and $\rho = \sigma\circ f_0$, and such that
$f = \Psi \circ \left(f_0\times f_1\right)$.
Equivalently, if the following diagram commutes.

\bigskip

\begin{picture}(150,84)\hspace{-15ex}
\put(180,30){${N_0}^{n_0}\,\times_{\rho}\, {N_1}^{n_1}$}
\put(192,55){$f_0$} \put(245,32){\vector(1,0){135}}
\put(385,30){$\Q^{n+p}_c$} \put(234,55){$f_1$}
\put(285,45){$\circlearrowright$}
\put(260,20){$f=\Psi\circ(f_0\times f_1)$}
\put(187,42){\vector(0,1){30}} \put(229,42){\vector(0,1){30}}
\put(178,80){$V^{n_0+l_0}\times_\sigma \mathbb{Q}^{n_1+l_1}_{\tilde{c}}$}
\put(255,80){\vector(3,-1){125}} \put(315,65){$\Psi$}
\end{picture}
\vspace*{-2ex}

\noindent If $l_1=0$, then $f$ is called a \textit{rotational submanifold} with profile $f_0$. We point out that this differs subtly from another in which it is referred to as a {\em hypersurface of revolution} in the case of the profile submanifold $f_0$ is a curve (cf. \cite[Definition 2.2]{doCarmoDajczerRotation}).
\end{definition}

The following lemma was established in \cite[Lemma 3.8]{PodestaSpiroCohomEuc} for the case in which the ambient space is the Euclidean space, and the same ideas can be applied to other space forms. Since the result will be used in the proof of Theorem \ref{thm:Rothyper}, we present a proof for the general case here. 

\begin{lemma} \label{lem:WarpedNull}
    Let $I\subset\R$ be an open interval, $N^{n-1}$ be a Riemannian manifold with constant section curvature $\tilde{c}$, and consider the warped product $M^n=I \times_{\rho} N^{n-1}$ with warping function $\rho$. Given a hypersurface $f\co M^n\to\mathbb{Q}^{n+1}_c$, where $c \in \R$ and $n \geq 2$, its shape operator $A$ satisfies one of the following conditions:
    \begin{enumerate}[(i)]
        \item $\rank\ A (p) \geq n-1$,
        \item $\rank\ A (p) \leq 1$,
    \end{enumerate}
    for each point $p\in M^n$. 
\end{lemma}
\begin{proof}
It is known that the curvature tensor of a warped product Riemannian manifold satisfy
    \begin{equation}\label{eq:warpedCurvature}
        \begin{split}    
        R^M(X,\xi)\xi &= - \frac{\text{Hess}_{\rho}(\xi,\xi)}{\rho} X,\\
        R^M(X,Y)\xi &= 0, \\
        R^M(X,Y)Z &= R^N(X,Y)Z - \frac{\Vert \nabla \rho \Vert^2}{\rho^2}( \langle Y,Z \rangle X - \langle X,Z \rangle Y),
        \end{split}
    \end{equation}
for all $X,Y,Z \in TN$, where $R^M$ and $R^N$ denote the curvature tensor of $M^m$ and $N^{m-1}$, respectively, and $\xi$ is a vector field tangent to the factor $I$. We will divide the proof in two cases.    
    First, suppose $A\xi = 0$ at a point $p$. This implies that $AX \in T_pN$ and $\rank\ A = \rank\ A|_{T_pN}$, where $ A|_{T_pN}$ is the restriction of $A$ to $T_pN$. From the Gauss equation, we have \[R^M(X,Y)Z = c \langle Y, Z \rangle X - \langle X, Z \rangle Y + \langle AY, Z \rangle AX - \langle AX, Z \rangle AY.\] Using \eqref{eq:warpedCurvature} and the fact that $N$ has constant sectional curvature $\tilde{c}$, it follows that
\[
\langle AY, Z \rangle AX - \langle AX, Z \rangle AY = 
\left(\tilde{c} - \frac{\Vert \nabla \rho \Vert^2}{\rho^2} \right) (\langle Y, Z \rangle X - \langle X, Z \rangle Y).
\] 
Thus, at the point $p$, we have $\rank\ A = n-1$ or $\rank\ A \leq 1$ depending on whether the expression $(\tilde{c} - \Vert \nabla \rho \Vert^2/\rho^2)$ is different to zero or equal to zero, respectively. Now, suppose that $A\xi \neq 0$. If $A\xi = \mu \xi$ at a point $p$, then by \eqref{eq:warpedCurvature} and the Gauss equation, we have
        \begin{equation*}
            \begin{split}
                - \frac{\text{Hess}_{\rho}(\xi,\xi)}{\rho} X &= R^M(X,\xi)\xi\\
                &= c(\langle \xi, \xi \rangle X - \langle X, \xi \rangle \xi) + \langle A\xi, \xi \rangle AX - \langle AX, \xi \rangle A\xi\\
                &= c\Vert \xi \Vert^2 X - \mu AX,
            \end{split}
        \end{equation*}
    for all $X \in T_pN$. This is equivalent to
\[
\mu AX = \left(c\Vert \xi \Vert^2 + \frac{\text{Hess}_{\rho}(\xi,\xi)}{\rho}\right)X.
\] 
Thus, $\rank\ A|_{T_pN} =n-1$ or $\rank\ A|_{T_pN} = 0$, i.e., $\rank\ A = n$ or $\rank\ A = 1$. Consider now the case $A\xi = \mu \xi + V_0$, with $0\neq V_0\in T_pN$ and $\mu \in \R$. Given $W \in \text{span}\{\xi,V_0\}^\perp$, using \eqref{eq:warpedCurvature} and the Gauss equation, we get 
\[
R^M (W,V_0)\xi = -\langle AW, \xi \rangle A V_0 + \langle A V_0, \xi \rangle A W,
\]
which implies 
\[
0 = \langle A\xi, W \rangle A V_0 = \langle V_0, V_0 \rangle AW.
\] 
Hence, span$\{V_0,\xi\}^{\perp} \subset \ker A$. Therefore, $\rank\ A \leq 2$. We will show now that $A V_0$ and $A\xi$ are linearly dependent. Since $n\geq 3$, taking $W \in \ker A \cap TN$, the warped metric and the Gauss equation give us
\[
-\frac{\text{Hess}_{\rho}(\xi,\xi)}{\rho} W = R^M(W,\xi)\xi = c\langle\xi,\xi\rangle W,
\] 
which implies
\[
-\frac{\text{Hess}_{\rho}(\xi,\xi)}{\rho} = c\Vert \xi \Vert^2
\]
at the point $p$. It follows from the expression of the curvature tensor $R^M$, applied to the vectors $V_0$ and $\xi$, that 
\[
\langle A\xi, \xi \rangle AV_0 - \langle AV_0, \xi \rangle A\xi =0.
\]
This implies that $\rank\ A =1$. This concludes the proof of the lemma.
\end{proof}

Finally, we present here the characterization of the objects studied in this paper when the principal orbits are umbilical and \(G\) is compact.

\begin{theorem}\label{thm:PrincipalsUmb}
    Let $f\co M^n \rightarrow \mathbb{H}^{n+1}$ be a complete hypersurface of $G$-cohomogeneity one whose the principal orbits are umbilical in $M^n$. Assume either that $n\geq 3$ and $M^n$ is compact or that $n\geq 4$ and the connected components of the set $\{x \in M^n: \nu(x)\geq n -1\}$ are bounded. Then f is a rotational submanifold.
\end{theorem}
\noindent The case where \( M^n \) is complete and \( n \geq 4 \) follows from \cite[Theorem 6.1]{AspertiCaputiCohomogeneityHyperbolic}. Note that they allowed \( G \) to be non-compact, which can lead to non-rotational examples. However, in this case, the orbits are necessarily non-compact. The case where the manifold is compact and \( n \geq 3 \) follows from the same arguments as in \cite[Corollary 4]{MoutinhoTojeiroUnico} (see also \cite[Proposition 2.2]{MercuriPodestaSeixasTojeiroCohomogeneityEuclidean}). These arguments involve a group homomorphism that naturally arises in this setting. Due to its inherent interest, we present the statement of this homomorphism for the case when the ambient space is also hyperbolic. The proof is the same as in the Euclidean case but uses Theorem \ref{thm:Sacksteder}.
\begin{theorem} \label{thm:ion}
    Let \( f\co M^n\to\Q^{n+1}_c \) be a complete hypersurface with \( n \geq 3 \) and \( c \leq 0 \). If there exists no complete leaf of relative nullity of dimension \( n-1 \) or \( n-2 \), then the identity component \( \Iso^0(M^n) \) of the isometry group of \( M^n \) admits an orthogonal representation \( \rho\co\Iso^0(M^n)\to\Iso^0(\Q^{n+1}_c) \) such that \( f\circ g=\rho(g)\circ f \) for all \( g\in\Iso^0(M^n) \).
\end{theorem}

\section{The proofs} \label{sec:Proofs}

\PfThm{\ref{thm:Rothyper}}
The beginning of the proof closely follows the proof of \cite[Proposition 2.3]{MercuriPodestaSeixasTojeiroCohomogeneityEuclidean}, establishing the objects to be used (such a structure was also provided in \cite[Proposition 3.1]{AspertiCaputiCohomogeneityHyperbolic}). Since $M^n$ carries a complete leaf of relative nullity $\mathcal{F}$, it can not be compact. Thus the orbit space $\Omega = M^n/ G$ is homeomorphic to either $\R$ or $[0,\infty)$. Moreover, if $\pi: M^n \rightarrow \Omega$ denotes the canonical projection and $\gamma: \R \rightarrow M^n$ is a normal geodesic parameterized by arc-length, then $\pi\circ \gamma$ maps $\R$ homeomorphically onto $\Omega$ in the first case, and it is a covering map of $\R\setminus\{0\}$ onto the subset $\Omega^0$ of internal points of $\Omega$ in the latter. Set $I = \gamma^{-1}(G(\mathcal{F}))$. Since $G(\mathcal{F})$ is a closed unbounded connected subset, and using that $G(\mathcal{F}) = G(\gamma(I))$, it follows easily that if $I \neq \R$ then $I=[a,\infty)$ for some $a \in \R$ in the first case and $I = (-\infty, b]\cup [a,\infty)$ for some $a,b >0$ in the latter. Note that the rank of the its shape operator is everywhere equal to two on $G(\mathcal{F})$, because the relative nullity subspace coincides with the nullity of the curvature tensor at a point where the rank is at least two, whence the subset where the type number is two is invariant under isometries. Let $p=\gamma(t_0)$ and consider $\epsilon>0$ such that $(t_0-\epsilon,t_0 +\epsilon)\subset I$ and such that the map
\[
\Phi: (t_0\epsilon, t_0 +\epsilon) \times \Sigma_p \rightarrow \pi^{-1}(t_0\epsilon, t_0 +\epsilon)
\]
given by
\[\Phi(t,g(p)) = g(\gamma(t)),
\]
is a $G$-equivariant diffeomorphism. We call $\Gamma = \pi^{-1}(t_0 - \epsilon, t_0 +\epsilon)$ a tube around $\Sigma_p$. We have a well-defined vector field $\xi$ on $\Gamma$ given by $\xi(y) = g_*(\gamma(t))\gamma'(t)$ for $y= g(\gamma(t))$, $t\in (t_0 \epsilon, t_0 +\epsilon)$, and $\xi(y)$ is orthogonal to $\Sigma_{\gamma(t)}$ at $y$.

Now let \(\eta\) be a local unit normal vector field to \(f\) on \(\Gamma\) and \(A^f_\eta\) the shape operator of \(f\) with respect to \(\eta\). Given a principal orbit \(\Sigma_q = G(q) \subset \Gamma\) of \(G\), the vector fields \(\tilde{\xi} = f_*({\xi |}_{\Sigma_q})\) and \(\tilde{\eta} = \eta|_{\Sigma_q}\) determine an orthonormal normal frame of the restriction \(\tilde{f}\co \Sigma_q \to \H^{n+1}\) of \(f\) to \(\Sigma_q\). Denote by \(A_{\tilde{\eta}}\) and \(A_{\tilde{\xi}}\) the corresponding shape operators (here we omit the superscript for clarity as there is no ambiguity). By construction we have $A_{\tilde{\xi}} \circ g_* = g_* \circ A_{\tilde{\xi}}$ for any $g \in G$, hence the eigenvalues of $A_{\tilde{\xi}}$ are constant. On the other hand, $A_{\tilde{\eta}}= \Pi \circ A^f_{\eta}$, where $\Pi$ is the orthogonal projection of $T M^n$ onto $T \Sigma_q$. In particular, $\text{rank} A_{\tilde{\eta}} \leq \text{rank} A^f_{\eta}$, so we have $\text{rank} A_{\tilde{\eta}} \leq 2$ on $\Sigma_q$. It follows from Gauss equation that we have the following two cases to consider

\begin{enumerate}
    \item\label{thmItem:Rank1} rank $A_{\Tilde{\eta}}\leq 1$ on each principal orbit contained in $\Gamma$;
    \item\label{thmItem:Rank2} rank $A_{\Tilde{\eta}} = 2$ on some principal orbit contained in $\Gamma$.
\end{enumerate}

First, we show that \eqref{thmItem:Rank1} does not occur. Assume the opposite. Since $\Sigma^{n-1}_p$ is a compact submanifold of $\H^{n+1}$, there is a point $x \in \Sigma^{n-1}_p$ where $\nu^{\tilde{f}}(x)=0$ (the index of relative nullity of \(\tilde{f}\)). From \cite[Theorem~6]{spaceFormsNoronhadeCastro}, the universal cover $\Tilde{\Sigma}^{n-1}_p$ of $\Sigma^{n-1}_p$ is isometrically immersed in $\H^n$ as a compact isoparametric hypersurface. This means $\Tilde{\Sigma}^{n-1}_p$ is isometric to $\S^{n-1}_c$ for some \(c>0\). Additionally, the shape operator \(A_{\tilde{\xi}}\) of $\Sigma^{n-1}_p$ will be a constant multiple of the identity. In particular, the principal orbits in $\Gamma$ are umbilical in $M^n$. Given that the integral curve of \(\xi\) is the normal geodesic, we derive from the main result of \cite{HiepnkoWarpedDeRham} that there is an open subset $U\subset M^n$, $U = I \times_{\rho} N^{n-1}_c$, where \(N\) is a Riemannian manifold with constant curvature \(c>0\), and $f(U) \subset \Gamma$ (refer also to Proposition 3.2 in \cite{AspertiCaputiCohomogeneityHyperbolic}). Considering $n \geq 5$ and $\nu^f|_{\Gamma} = n-2$, this leads to a contradiction with Lemma \ref{lem:WarpedNull}.

If \(\rank\ A_{\tilde{\eta}} = 2\) along a principal orbit \(\Sigma^{n-1}_p \subset \Gamma\), then the same rank condition holds on a smaller tube around \(\Sigma^{n-1}_p\) within \(\Gamma\), which we continue to denote by \(\Gamma\). Our goal now is to obtain a decomposition of the tangent space of \(M^n\) compatible with a warped structure as stated in the theorem, in order to use the main result in \cite{NolkerWarped}. To achieve this, we need to show that \(\nu^f\) is constant throughout \(M^n\) and that the eigenspaces of \(A_{\tilde{\xi}}\) are autoparallel distributions in \(\Sigma^{n-1}_p\). This will follow from the series of assertions proved below. In the proof, \(\nabla\) will denote the connection on \(\Sigma^{n-1}_p\) and \({}^{\scriptscriptstyle M}\nabla\) will denote the connection on \(M^n\).

\begin{claim}
The vector field \(\tilde{\eta}\) is parallel along the distribution \(\ker A_{\tilde{\eta}}\) with respect to the normal connection of \(\tilde{f}\), represented here simply by \(\nabla^{\perp}\).
\end{claim}
We know that \(A_{\tilde{\xi}}\) is \(G\)-equivariant. If \(A_{\tilde{\eta}}\) is also \(G\)-equivariant, then \(\Sigma^{n-1}_p\) is an extrinsically homogeneous submanifold, and we will consider this first case. If \(\Sigma^{n-1}_p\) admits a reduction in codimension, meaning it is not full, it is a hypersurface of a totally geodesic \(\H^{n}\) in \(\H^{n+1}\), leading to a contradiction with the value of the index of relative nullity by the same arguments as in case \eqref{thmItem:Rank1}. Therefore, \(\Sigma^{n-1}_p\) is a compact submanifold of codimension two, which must be full. According to the theory of isoparametric submanifolds (cf. \cite{BerndtConsoleOlmosBook}), it has to be a homogeneous isoparametric hypersurface of a sphere \(\S^n_{c}\), for some \(c>0\), which is umbilically embedded in \(\mathbb{H}^{n+1}\).
In particular, at every point of \(\Sigma^{n-1}_p\), there exist \(\lambda \neq 0, a, b \in \mathbb{R}\) such that
\[ 
\lambda I = a A_{\tilde{\eta}} + b A_{\tilde{\xi}},
\]
where \(I\) is the matrix identity, and since the rank of \(A_{\tilde{\eta}}\) is 2 and \(n \geq 5\), it follows that \(b \neq 0\) and, consequently, \({A_{\tilde{\xi}}}|_{\ker A_{\tilde{\eta}}}\) is a non-zero multiple of the identity. Now consider the Codazzi equation
\[
A_{\nabla^{\perp}_Y \tilde{\eta}} X + A_{\tilde{\eta}}(\nabla_Y X) = A_{\nabla^{\perp}_X \tilde{\eta}} Y + A_{\tilde{\eta}}(\nabla_X Y),
\]
for vectors \(Y, X \in \ker A_{\tilde{\eta}}\). Since \(n \geq 5\), for each \(X\), we can find \(Y\) orthogonal to \(X\). Taking the inner product of the previous equation with \(Y\), we obtain the following relation:
\[
\langle \nabla^{\perp}_X \tilde{\eta}, \tilde{\xi} \rangle \langle A_{\tilde{\xi}} Y, Y \rangle = 0,
\]
which implies the parallelism of \(\tilde{\eta}\) along \(\ker A_{\tilde{\eta}}\) in the case where \(\Sigma^{n-1}_p\) is extrinsically homogeneous. In the case where it is not extrinsically homogeneous, there exists \(g \in G\) such that \(A_{\tilde{\eta}} g_* \neq \pm g_* A_{\tilde{\eta}}\) everywhere. Using the same argument as in \cite[Lemma 6]{DajczerGromollIsoCod2}, it follows that \(\tilde{\eta}\) is parallel along \(\ker A_{\tilde{\eta}}\).
\EPfC

Let \(h\co N_p^{n-3} \rightarrow \H^{n+1}\) be the restriction of \(\tilde{f}\) to the leaf \(N_p^{n-3}\) of \(\ker A_{\tilde{\eta}}\) passing through the point \(p\). It is not difficult to see that \(\ker A_{\tilde{\eta}}\) is \(G\)-invariant and the leaves of the distribution are homogeneous submanifolds, so we will omit the point in the notation. Therefore, if one leaf of the distribution \(\ker A_{\tilde{\eta}}\) is autoparallel, then all are.

\begin{claim}
    \(\ker A_{\tilde{\eta}}\) is an autoparallel distribution in \(\Sigma^{n-1}_p\).
\end{claim}
Consider the Codazzi equation
\[
- A_{\nabla^{\perp}_Z \tilde{\eta}} X - A_{\tilde{\eta}}(\nabla_Z X) = \nabla_X A_{\tilde{\eta}} Z - A_{\nabla^{\perp}_X \tilde{\eta}} Z - A_{\tilde{\eta}}(\nabla_X Z),
\]
for \(X \in \ker A_{\tilde{\eta}}\) and \(Z \in \text{Im} A_{\tilde{\eta}}\). Taking the inner product with \(Y \in \ker A_{\tilde{\eta}}\) and using the parallelism of \(\tilde{\eta}\) along \(\ker A_{\tilde{\eta}}\), we obtain
\begin{equation}\label{eq:CodazziG}
    \langle \nabla^{\perp}_Z \tilde{\eta}, \tilde{\xi} \rangle \langle \alpha_{\tilde{f}}(X, Y), \tilde{\xi} \rangle = \langle \nabla_X Y, A_{\tilde{\eta}} Z \rangle.
\end{equation}
Define \(\omega\) as the one-form where \(\omega(X) = \langle \nabla^{\perp}_X \tilde{\xi}, \tilde{\eta} \rangle\) for \(X \in T\Sigma\). Consider two local nonzero vector fields \(Z_1, Z_2 \in \Gamma(T\Sigma)\) such that \(\Vert A_{\tilde{\eta}} Z_1\Vert = \Vert A_{\tilde{\eta}} Z_2\Vert = 1\), \(A_{\tilde{\eta}} Z_1 \perp A_{\tilde{\eta}} Z_2\), and \(Z_2 \in \ker \omega\). From \eqref{eq:CodazziG}, the second fundamental form of \(h\), denoted by \(\alpha_h\), evaluated on \(X, Y \in TN\) is
\begin{equation}\label{eq:SecFundgGeneral}
    \alpha_h (X, Y) =\begin{cases}
       \langle A_{\tilde{\xi}} X, Y \rangle (\tilde{\xi} + \langle \nabla^{\perp}_{Z_1} \tilde{\eta}, \tilde{\xi} \rangle A_{\tilde{\eta}} Z_1) & \text{if} \ \ker A_{\tilde{\eta}} \text{ is not autoparallel,} \\
       \langle A_{\tilde{\xi}} X, Y \rangle \tilde{\xi}  & \text{if} \ \ker A_{\tilde{\eta}} \text{ is autoparallel.} 
    \end{cases}
\end{equation}
Therefore, the first normal bundle of \(h\) is one-dimensional. Given \(n \geq 5\), we apply \cite[Theorem~A]{sNullitiesReductionDajczerRodriguez} (or \cite[Proposition 2.7]{base}) to obtain a reduction of codimension for \(h\co N^{n-3} \rightarrow \H^{n-2}\). Since the image of \(h\) is bounded, the classification of homogeneous hypersurfaces implies that \(h\) is umbilical and \(N^{n-3}\) is isometric to a sphere \(\S^{n-3}_{c_2}\) for some constant \(c_2>0\). Since \(N^{n-3}\) is not isometric \(\R^{n-3}\), from \cite[Lemma~17]{spaceFormsNoronhadeCastro}, \(\ker A_{\tilde{\eta}}\) must be autoparallel.
\EPfC

Our goal now is to show that the shape operator \(A_{\tilde{\xi}}\) has only two distinct eigenvalues, and the associated eigenspaces will give rise to the distributions that allow us to obtain the warped structure mentioned in the theorem statement. 

\begin{claim}
\(\xi \in \Delta^f\), \(\nu^f = n-2\), and the operator \(A_{\tilde{\xi}}\) has only two distinct eigenvalues, with the eigenspaces being autoparallel distributions in \(\Sigma^{n-1}_p\).
\end{claim}

Using the umbilicity of \(h\), \eqref{eq:SecFundgGeneral} simplifies to
\[
\alpha_h(X,Y) = \langle A_{\tilde{\xi}} X,Y \rangle {\tilde{\xi}} = \lambda \left \langle X,Y \right \rangle {\tilde{\xi}} \quad \text{for all}\, X,Y \in TN,
\]
for some constant \( \lambda > 0\). According to the Gauss equation, \(c = -1 + \lambda^2\). Furthermore, \(\ker A_{\tilde{\eta}}\) being autoparallel implies from \eqref{eq:CodazziG} that either \(\langle \nabla^{\perp}_W \tilde{\eta}, \tilde{\xi} \rangle = \langle A_{\tilde{\eta}}W, \tilde{\xi} \rangle = 0\) or \(\langle A_{\tilde{\xi}}X, Y \rangle = 0\) for all \(X, Y \in \ker A_{\tilde{\eta}}\). The latter case would imply that \(h\) is totally geodesic, contradicting the compactness of \(\Sigma^{n-1}_p\). Hence, \(\langle A_{\tilde{\eta}}W, \tilde{\xi} \rangle = 0\) together with the parallelism of \(\tilde{\xi}\) and \(\tilde{\eta}\) along \(\ker A_{\tilde{\eta}}\) implies \(\tilde{\xi} \in \Delta^f(q)\) for all \(q \in \Sigma^{n-1}_p\). Since \(\tilde{\xi}\) is the restriction of the field generated by the normal geodesic \(\xi\), it follows that \(\xi \in \Delta^f\) everywhere. Therefore, the segments of normal geodesics in \(\Gamma\) are contained in the leaves of \(\Delta^f = \ker A_\eta\). Given that these leaves are assumed to be complete, we obtain that \(\nu^f = 2\) throughout \(M^n\).

Since \(h\) has a flat normal bundle, the shape operators \(A_{\tilde{\xi}}\) and \(A_{\tilde{\eta}}\) can be diagonalized simultaneously with a basis \(\{X_1, X_2, X_3, \dots, X_{n-1}\}\), where \(X_1, X_2\) generate \((\ker A_{\tilde{\eta}})^\perp\). Let \(\delta_1, \delta_2\) be the eigenvalues of \(A_{\tilde{\eta}}\), and \(\lambda_i\) the constant eigenvalues of \(A_{\tilde{\xi}}\). The Codazzi equation for \(X_1, X_2\) in the direction \(\tilde{\eta}\) gives
\[
\left\langle \nabla_{X_1} X_2, X \right\rangle \delta_2 = \left\langle \nabla_{X_2} X_1, X \right\rangle \delta_1,
\]
for all \(X \in \ker A_{\tilde{\eta}}\). Since \(\delta_i \neq 0\), the orthogonal projections of \(\nabla_{X_1} X_2, \nabla_{X_2} X_1\) onto \(\ker A_{\tilde{\eta}}\) are linearly dependent. From the previous discussion on umbilicity, it follows that \(\lambda_i = \lambda\) for \(i \geq 3\). The goal is to show that \(\lambda_1 = \lambda_2 \neq \lambda\). If \(\lambda_1 = \lambda_2 = \lambda\), the distribution \((\text{span}\{\xi\})^\perp\) is spherical in \(M^n\), since at each point \(p \in M^n\), \((\text{span}\{\xi\})^\perp\) is identified with the tangent space of the orbit \(\Sigma^{n-1}_p\). It follows from \cite{HiepnkoWarpedDeRham} that \(M^n = I \times_{\rho} \S^{n-1}_{c}\) for some constant \(c > 0\), contradicting the dimension of relative nullity as stated in Lemma \ref{lem:WarpedNull}. Assuming \(\lambda_1 \neq \lambda_2\), we may, without loss of generality, assume \(\lambda_1 \neq \lambda\). Thus, \(\left \langle \nabla_{X_1} X_1, X \right \rangle = 0\) for all \(X \in \ker A_\eta\). Lemma 10.b in \cite{spaceFormsNoronhadeCastro} implies that
\[
\left \langle \nabla_{X_1} X_1 + \nabla_{X_2} X_2, X \right \rangle = 0
\]
for all \(X \in \ker A_\eta\). Together with \(\left \langle \nabla_{X_1} X_1, X \right \rangle = 0\), this leads to \(\left \langle \nabla_{X_i} X_i, X \right \rangle = 0\) for \(i \in \{1, 2\}\). Choosing \(X \in \ker A_{\tilde{\eta}}\) orthogonal to the projections \(\nabla_{X_1} X_2, \nabla_{X_2} X_1\), we conclude
\[
\left \langle \nabla_X \nabla_{X_i} X_i, X \right \rangle = 0,
\]
for \(i \in \{1, 2\}\), since \(\left \langle \nabla_{X_i} X_i, X \right \rangle = 0\) and \(\ker A_{\tilde{\eta}}\) is an autoparallel distribution,
\[
\left \langle \nabla_{X_i} \nabla_X X_i, X \right \rangle = 0,
\]
for \(i \in \{1, 2\}\), as \(\text{Im} A_{\tilde{\eta}}\) is parallel along \(\ker A_{\tilde{\eta}}\) and by our choice of \(X\),
\[
\left \langle \nabla_{[X, X_i]} X_i, X \right \rangle = 0,
\]
for \(i \in \{1, 2\}\), using all the previous information. It follows that
\[
\left \langle R(X, X_i) X_i, X \right \rangle = 0
\]
for \(i \in \{1, 2\}\). From the Gauss equation, we have \(0 = K(X, X_i) = -1 + \lambda_i \langle A_{\tilde{\xi}} X, X \rangle\), \(i \in \{1, 2\}\), contradicting \(\lambda_1 = \lambda_2\). Now we have \(\lambda_0 \coloneqq \lambda_1 = \lambda_2\), and \(A_{\tilde{\xi}}\) is expressed in this basis as
\begin{equation}\label{eq:ShapeOperators}
A_{\tilde{\xi}} = 
\left( \begin{array}{@{}c|c@{}}
   \begin{matrix}
    \lambda_0 & 0 \\
    0 & \lambda_0
   \end{matrix}
      & 0 \\
   \cline{1-2}
   0 & \lambda I \\
\end{array} \right).
\end{equation}
A direct application of the Codazzi equation, together with the constancy of \(\lambda_0\) and \(\lambda\), implies that \(E_{\lambda_0}\) and \(E_{\lambda}\) are autoparallel distributions in \(\Sigma^{n-1}_p\).
\EPfC

We thus have the natural decomposition \(TM = \text{span}\{\xi\} \oplus V_0 \oplus V_1\) at the point \(q \in M^n\), where \(V_0 = (\Delta^f)^\perp\) and \(V_1\) is the orthogonal complement of \(\text{span}\{\xi\}\) in \(\Delta^f\). Since the expression \eqref{eq:ShapeOperators} holds for all principal orbits of \(G\), it follows that \(V_0\) at the point \(q \in M^n\) is identified with \(E_{\lambda_0}\) and is therefore an umbilic distribution whose mean curvature vector is parallel to \(\xi\) with constant norm along the principal orbits of \(G\), making it a spherical distribution in \(M^n\). Similarly, \(V_1\) is also a spherical distribution in \(M^n\). Summarizing, it follows from the claims proven above that
\[
{}^{\scriptscriptstyle M}\nabla_{\xi} \xi = 0,
\]
\[
{}^{\scriptscriptstyle M}\nabla_{X_i} Y_i \subset V_i \oplus \text{span}\{\xi\},
\]
\[
{}^{\scriptscriptstyle M}\nabla_{\xi} Y_i \subset V_i,
\]
and 
\[
{}^{\scriptscriptstyle M}\nabla_{X_i} \xi \subset V_i
\]
for all \(X_i, Y_i \in V_i\), \(i \in \{0,1\}\). Therefore, \(V_0\) and \(V_1\) are spherical distributions in \(M^n\), whose orthogonal complements are autoparallel distributions in \(M^n\). It follows from the results of \cite{NolkerWarped} (see also \cite[Theorem 10.21]{base}) that \(M^n\) is isometric to \(\R \times_{\rho_0} \S^2_{c_1} \times_{\rho_1} \S^{n-3}_{c_2}\), for some constants \(c_1 > 0\) and \(c_2 > 0\), with the warped functions \(\rho_0\) and \(\rho_1\) defined on \(\R\), and \(f\) admits an extrinsic warped product structure. More precisely, there exists a warped product representation \(\Psi\co V^{1+i_0} \times_{\sigma_0} \mathbb{Q}^{2+i_1}_{\tilde{c_1}} \times_{\sigma_1} \S^{n-3+i_2}_{c_2} \rightarrow \mathbb{H}^{n+1}\), where \(i_0 + i_1 +i_2= 1\) and \(V^{1+i_0}\) is an open subset of a totally geodesic submanifold, with warping functions \(\sigma_1\) and \(\sigma_2\) defined on \(V^{1+i_0}\). Furthermore, there exist isometric immersions \(f_0\co \mathbb{R} \rightarrow V^{1+i_0}\), \(f_1\co \S^2_{c_1} \rightarrow \mathbb{Q}^{2+i_1}_{\tilde{c_1}}\), and \(f_2\co \S^{n-3}_{c_2} \rightarrow \mathbb{Q}^{n-3+i_2}_{\tilde{c_2}}\) such that \(\rho_1 = \sigma_1 \circ f_0\), \(\rho_2 = \sigma_2 \circ f_0\), and \(f = \Psi \circ (f_0\times f_1\times f_2)\). In particular, the following diagram commutes:
\bigskip

\begin{picture}(150,84)\hspace{-15ex}
\put(140,30){\(\R \ \ \times_{\rho_0} \ \ \ \S^2_{c_1}\ \  \times_{\rho_1} \S^{n-3}_{c_2}\)}
\put(150,55){$f_0$} \put(192,55){$f_1$} \put(245,32){\vector(1,0){135}}
\put(385,30){$\H^{n+1}$.} \put(234,55){$f_2$}
\put(285,45){$\circlearrowright$}
\put(260,20){$f=\Psi\circ(f_0\times f_1 \times f_2)$}
\put(145,42){\vector(0,1){30}} \put(187,42){\vector(0,1){30}} \put(229,42){\vector(0,1){30}}
\put(128,80){\(V^{1+i_0} \times_{\sigma_0} \mathbb{Q}^{2+i_1}_{\tilde{c_1}} \times_{\sigma_1} \mathbb{Q}^{n-3+i_2}_{\tilde{c_2}}\)}
\put(255,80){\vector(3,-1){125}} \put(315,65){$\Psi$}
\end{picture}
\vspace*{-2ex}

\noindent If \(i_2 = 1\), using that \(T\S^{n-3}_{c_2} \subset \Delta^f\), it is straightforward to verify from the curvature tensor equations of a warped product that \(\tilde{c_2} = c_2 > 0\), which implies that \(f_2\) is a totally geodesic immersion (cf. \cite[Corollary 7.13]{base}). As a result, \(f\) would be totally geodesic. Consequently, \(i_2 = 0\) and \(i_0 + i_1 = 1\). This places us precisely in the scenario of the theorem's statement. In particular, we have a rotational submanifold whose profile submanifold is \(N^3 = \R \times_{\rho_0} \S^2_{c_1}\) and the immersion is determined by \(f_0, f_1\) and the warped product representation \(\Psi\), thus concluding the proof.
\EPf

\begin{remark}
In the proof of Theorem \ref{thm:Rothyper}, we already establish that the submanifold is isometric to \(\mathbb{H}^{n-2} \times \mathbb{S}^2_{c_1}\), but not yet that it is a rotational submanifold when we show that the relative nullity is spherical. However, the presented proof goes further by showing that the profile submanifold is foliated by spheres and that the principal orbits of \(G\) are actually products of spheres.
\end{remark}

\PfThm{\ref{thm:Main}}
    It follows from our assumption on the set $\{x \in M^n: \nu(x)\geq n -1\}$ that does not exist a complete leaf of relative nullity with \(\nu^f = n - 1\) or \(\nu^f = n\). If there exists a complete leaf of relative nullity with \(\nu^f = n - 2\), then by Theorem \ref{thm:Rothyper}, \(M^n\) must be a rotational submanifold. Thus, we can assume that there are no complete leaves of relative nullity with \(\nu^f > n - 3\).
    %putting us in a position to apply Theorem \ref{thm:Sacksteder}.
    Let \(B^f\) denote the set of totally geodesic points of \(f\). According to Theorem \ref{thm:Sacksteder}, \(f\) is either rigid or \(B^f\) disconnects \(M\). We will show that if we are in the latter case, \(f\) is a rotational submanifold. Fix a point \(p \in B^f \cap M_{\mathrm{reg}}\), where \(M_{\mathrm{reg}}\) stands as the regular points of $M^n$ with the respect to the action, and consider the orbit \(\Sigma^{n-1}_p = G(p)\). By applying Theorem \ref{thm:Sacksteder} to each pair $f$ and $f\circ g$, where $g\in G$, we find
    \[
    A^f = \pm A^{f \circ g},
    \]
    for all \(g \in G\). Moreover, at \(q \in \Sigma^{n-1}_p\), for all \(g \in G\), it holds that
    \[
    \langle \nabla_{f_* g_* Y}f_* g_* X, \eta \circ g \rangle = 
    \langle A^{f \circ g}_{\eta \circ g}X, Y \rangle = \langle A^f_\eta g_* X, g_* Y \rangle
    \]
    for \(X, Y \in T_q\Sigma\) and \(\eta \in T_f^\perp M(q)\). This leads to 
    \(A^{f \circ g}_{\eta \circ g} = g_*^{-1}A^f_\eta g_*\), for all \(g \in G\), showing that \(B^f\) is \(G\)-invariant. Considering the composition \(\Sigma^{n-1}_p \xrightarrow{i} M^n \xrightarrow{f} \H^{n+1}\), where $i$ is the inclusion, and since \(\Sigma^{n-1}_p \subset B^f\), the second fundamental form is described by
    \begin{equation}\label{eq:Secfleaf}
        \alpha_{f \circ i}(X,Y) = f_*\alpha_i(X,Y) + \alpha_f(i_*X,i_*Y) = f_*\alpha_i(X,Y),
    \end{equation}
    for \(X, Y \in T\Sigma\). Given the construction of \(\Sigma^{n-1}_p\), \(A^{f \circ i}_{f_*\xi}\) is \(G\)-invariant for a unit vector field \(\xi \in T^\perp_i \Sigma\). The Codazzi equation for vector \(X, Y \in T\Sigma\) in the direction of \(\tilde{\xi} = f_*\xi \in T^\perp_f M\) simplifies to
    \begin{equation}\label{eq:Omega}
        \omega(X) A_{\eta} Y = \omega(Y) A_{\eta} X,
    \end{equation}
    where \(\omega(X) = \langle \nabla^\perp_{X} \tilde{\xi}, f_*\eta \rangle\) for \(X \in T\Sigma\). The strategy will be to show that \(\Sigma^{n-1}_p\) is umbilical in \(M^n\). To achieve this, we prove that \(\tilde{\xi}\) is parallel along \(\Sigma^{n-1}_p\), meaning \(T\Sigma_p = \ker \omega\). Assuming otherwise and letting \(x \in \Sigma^{n-1}_p\) with \(\ker \omega \neq T\Sigma_p\), \eqref{eq:Omega} implies \(\ker \omega \subset \ker A_\eta\), leading to \(\nu_{f \circ i}(x) \geq n-2 \geq 2\). With \(\ker A^{f\circ i}_{f_*\eta}\) being \(G\)-invariant, \(\nu_{f \circ i}(x) \geq 2\) for all \(x \in \Sigma^{n-1}_p\) contradicts the compactness of \(\Sigma^{n-1}_p\) as complete leaves of relative nullity are unbounded. Hence, \(\ker \omega = T\Sigma_p\). From \eqref{eq:Secfleaf}, the first normal bundle \( \mathcal{N}^f_1 \coloneqq \text{span}\{\alpha(X,Y): X,Y \in TM\} \) of \( f\circ i \) forms a one-dimensional distribution, and according to the main result in \cite{sNullitiesReductionDajczerRodriguez} (or Proposition 2.7 in \cite{base}), this distribution is parallel, indicating \( f \circ i\) lies within a totally geodesic submanifold \(\H^{n}\) of \(\H^{n+1}\). Given the compactness and classification of homogeneous hypersurfaces in hyperbolic space, \(\Sigma^{n-1}_p\) is isometric to \(\S^{n-1}_{c_1}\) and the isometric immersion \(f\circ i\) is totally umbilical. This, along with \eqref{eq:Secfleaf}, concludes \(\Sigma^{n-1}_p\) is intrinsically umbilical. As the principal orbits are totally umbilical, by Theorem \ref{thm:PrincipalsUmb}, \(f\) is a rotational submanifold.
\EPf

\begin{remark}
    Although there are versions of Theorem \ref{thm:Sacksteder}, and consequently of Theorem \ref{thm:ion}, where the ambient space is the sphere, the principal orbits of \( G \) will not be intrinsically umbilical. Instead, they will be isoparametric hypersurfaces of a totally geodesic hypersurface of the ambient space. In this particular case, we cannot guarantee that the principal orbits of \( G \) will be mapped by the natural homomorphism to principal orbits of \( \Phi(G) \) (cf. \cite[Remark 12]{MoutinhoTojeiroUnico}).
\end{remark}

\section*{Acknowledgements}
Felippe Guimar\~aes is supported by the Para\'iba State Research Support Foundation (FAPESQ/PB) and partially by the Brazilian National Council for Scientific and Technological Development (CNPq), grant 409513/2023-7. Fernando Manfio is supported by the S\~ao Paulo Research Foundation (FAPESP), grant 2022/16097-2.

\bibliographystyle{abbrv}
\bibliography{bibliography}

\vskip 0.2cm

\noindent Felippe Guimarães

\noindent Departamento de Matematica, \\ Universidade Federal da Paraíba, João Pessoa (Brazil).

\noindent  e-mail: {\tt fsg@academico.ufpb.br}

\vskip 0.2cm

\noindent Fernando Manfio

\noindent Instituto de Ciências Matemáticas e de Computação, \\ Universidade de São Paulo, São Carlos (Brazil).

\noindent  e-mail: {\tt manfio@icmc.usp.br}

\vskip 0.2cm

\noindent Carlos E. Olmos

\noindent Facultad de Matem\'atica, Astronomía, F\'isica y Computaci\'on \\ Universidad Nacional de C\'ordoba, C\'ordoba (Argentina).

\noindent  e-mail: {\tt olmos@famaf.unc.edu.ar}

\end{document}